\documentclass[12pt,reqno]{amsart}
\usepackage{amsmath, amsthm, amssymb}
\usepackage{mathtools}
\usepackage{mathrsfs}  
\usepackage{stmaryrd}
\usepackage{comment}
\usepackage{marginnote}

\usepackage[margin=1in]{geometry}
\usepackage{parskip}
\usepackage[shortlabels]{enumitem}

\usepackage{xcolor}

\usepackage[bookmarks,colorlinks,breaklinks]{hyperref}
\hypersetup{linkcolor=red,citecolor=blue,
  filecolor=dullmagenta,urlcolor=blue}

\usepackage[numbers]{natbib}
\usepackage[capitalise]{cleveref}

\usepackage{graphicx}
\usepackage{tikz}
\usetikzlibrary{matrix, arrows}
\usetikzlibrary{cd}
\usepackage[all]{xy}
   \SelectTips{cm}{10}

\usepackage{hyperref}
\usepackage{enumitem}
\usepackage{graphicx}

\makeatletter
\def\thm@space@setup{%
  \thm@preskip=0.5em\thm@postskip=\thm@preskip%
}
\makeatother

\newtheoremstyle{named}{}{}{\\itshape}{}{\bfseries}{.}{.5em}{\thmnote{#3's }#1}
\theoremstyle{named}

\theoremstyle{plain}
\newtheorem{thm}{Theorem}[section]

\newtheorem{lem}[thm]{Lemma}

\newtheorem{cor}[thm]{Corollary}
\theoremstyle{definition}

\theoremstyle{remark}
\newtheorem{rmk}[thm]{Remark}
\crefformat{footnote}{#2\footnotemark[#1]#3}

\newcommand{\Hom}{\mathrm{Hom}}

\newcommand{\CC}{\mathbb{C}}
\newcommand{\QQ}{\mathbb{Q}}
\newcommand{\RR}{\mathbb{R}}

\newcommand{\ZZ}{\mathbb{Z}}

\newcommand{\tr}{\mathrm{tr}}

\newcommand{\Gal}{\mathrm{Gal}}

\newcommand{\mc}[1]{\mathcal{#1}}

\newcommand{\mr}[1]{\mathrm{#1}}

\newcommand{\ol}[1]{\overline{#1}}

\newcommand{\wt}[1]{\widetilde{#1}}

\newcommand{\Spec}{\mathrm{Spec}}

\newcommand{\Rep}{\mathrm{Rep}}

\newcommand{\Fet}{\mathrm{F\acute{E}t}}
\DeclareMathOperator{\FIso}{F-Isoc}

\newcommand{\Frob}{\mathrm{Frob}}

\newcommand{\Es}{E(sK_0)}

\DeclareMathOperator{\Aut}{Aut}

\DeclareMathOperator{\Pross}{Pro-ss}
\DeclareMathOperator{\Prored}{Pro-red}

\newcommand{\git}{{\,\!\sslash\!\,}}

\newcommand{\Pired}{\hat{\Pi}_{\lambda}^{\mr{red}}}
\newcommand{\Piplain}{\hat{\Pi}_{\lambda}^{\mr{plain}}}
\newcommand{\Pimot}{\hat{\Pi}_{\lambda}^{\mr{mot}}}

\usepackage[normalem]{ulem}
\usepackage[utf8]{inputenc}

\title{Compatibility of canonical $\ell$-adic local systems on Shimura varieties, II}

\thanks{I thank Christian Klevdal and Jake Huryn for discussions related to this material and for their comments on an earlier version.}

\begin{document}

\author[S.~Patrikis]{Stefan Patrikis}
\address{Department of Mathematics, The Ohio State University\\ 100 Math Tower\\ 231 West 18th Avenue\\ Columbus, OH 43210, USA}
\email{patrikis.1@osu.edu}

\begin{abstract}
Let $(G, X)$ be a Shimura datum. In previous work \cite{klevdal-patrikis:SVcompatiblearXiv} with Klevdal, we showed that the canonical $G(\QQ_{\ell})$-valued local systems on Shimura varieties for $G$ form compatible systems after projection to the adjoint group of $G$. In this note, we strengthen this result to prove compatibility for the $G(\QQ_{\ell})$-local systems themselves. We also include the crystalline compatibility, extending the adjoint case established in our joint work \cite{huryn-kedlaya-klevdal-patrikis} with Huryn, Kedlaya, and Klevdal.
\end{abstract}
\maketitle

\section{Introduction}

Let $(G, X)$ be a Shimura datum, and let $K_0 \subset G(\mathbb{A}_f)$ be a neat compact open subgroup. The associated Shimura variety $\mr{Sh}_{K_0}:= \mr{Sh}_{K_0}(G, X)$ is a smooth quasi-projective variety over the reflex field $E(G, X)$, with underlying complex manifold $\mr{Sh}_{K_0}(G, X)(\CC)= G(\QQ) \backslash (X \times G(\mathbb{A}_f))/K_0$. Varying the compact open subgroup $K \subset K_0$ we have a morphism of varieties over $E(G, X)$, $\mr{Sh}_K \to \mr{Sh}_{K_0}$, and we write $\mr{Sh}= \lim_{K} \mr{Sh}_K$. Fix a basepoint $s \in \mr{Sh}(\CC)$, inducing basepoints in all finite-level Shimura varieties $\mr{Sh}_K$ as well. We then obtain the geometrically connected component $S_{K, s}$ of $\mr{Sh}_{K}$ containing $s$; it is defined over some finite (abelian) extension $E_{K, s}$ of $E(G, X)$. When $Z_G(\QQ) \subset Z_G(\mathbb{A}_f)$ is a discrete subgroup, for any normal compact open subgroup $K \subset K_0$, $\mr{Sh}_K \to \mr{Sh}_{K_0}$ is a Galois cover with Galois group $K_0/K$, and (via the choice of $s$) these assemble to canonical local systems
\[
\pi_1(S_{K_0, s}, s) \to K_0, 
\]
which for varying $\ell$ give us by projection the canonical $\ell$-adic local systems
\[
\rho_{\ell} \colon \pi_1(S_{K_0, s}, s) \to G(\QQ_{\ell}).
\]
A long-standing hope (\cite{deligne:canonicalmodels}) is that $\mr{Sh}_{K_0}$ parametrizes a family of motives; this gives rise to the expectation that the $\rho_{\ell}$ form a compatible system of $\ell$-adic representations, at least in the following sense:
\begin{enumerate}
\item There is an integer $N$ (independent of $\ell$), and an integral model $\mc{S}_{K_0, s}$ over $\mc{O}_{E_{K_0, s}}[1/N]$, such that for all $\ell$, $\rho_{\ell}$ extends to (factors through) an arithmetic local system
\[
\rho_{\ell} \colon \pi_1(\mc{S}_{K_0, s}, s) \to G(\QQ_{\ell}).
\]
\item For all $\ell$, for all closed points $x$ of (the finite-type $\ZZ$-scheme) $\mc{S}_{K_0, s}[1/\ell]$, the semisimple conjugacy class underlying $\rho_{\ell}(\Frob_x)$ is defined over $\ol{\QQ}$ and is independent of $\ell$.
\end{enumerate}
For more refined expectations, see the end of this introduction. Much is known about problems (1) and (2). When $(G, X)$ is of abelian type, these and much more are known by work of Kisin (\cite{kisin:intmodabelian}, \cite{kisin:modpabelian}). When $(G, X)$ is not of abelian type---let us for simplicity at present assume that $G^{\mr{ad}}$ is $\QQ$-simple---then $\mr{rk}_{\RR} (G^{\mr{ad}}_{\RR}) \geq 2$, and our joint work \cite{klevdal-patrikis:SVcompatiblearXiv} with Klevdal proved that (1) and (2) hold after projecting to the adjoint local systems $\rho_{\ell}^{\mr{ad}} \colon \pi_1(S_{K_0, s}, s) \to G^{\mr{ad}}(\QQ_{\ell})$.

The present paper resolves (1) and (2) in general, establishing also some refinements in this generality: $\QQ$-rationality of the independent of $\ell$-conjugacy classes in (2), and crystalline compatibility, which builds on the analogous result in the adjoint case in our joint work \cite{huryn-kedlaya-klevdal-patrikis} with Huryn, Kedlaya, and Klevdal. 
\begin{thm}\label{mainthmintro}
Let $(G, X)$ be a Shimura datum such that $Z_G(\QQ)$ is a discrete subgroup of $Z_G(\mathbb{A}_f)$, let $K_0 \subset G(\mathbb{A}_f)$ be a neat compact open subgroup, and let $s \in \mr{Sh}(\CC)$, $S_{K_0, s}$ be as above. Assume that for all $\QQ$-simple factors $H$ of $G^{\mr{ad}}$, $\mr{rk}_{\RR}(H_{\RR}) \geq 2$. Then there is an integer $N$ and an integral model $\mc{S}:= \mc{S}_{K_0, s}$ over $\mc{O}_{E_{K_0, s}}[1/N]$ such that 
\begin{enumerate}
\item For all $\ell$, $\rho_{\ell}$ factors through 
\[
\pi_1(\mc{S}, s) \to G(\QQ_{\ell}).
\]
\item For all closed points $x \in \mc{S}_{K_0, s}[1/\ell]$, the class of $\rho_{\ell}(\Frob_x)$ in $[G \git G](\ol{\QQ}_{\ell})$ lies in $[G \git G](\QQ)$ and is independent of $\ell$ (not equal to the residue characteristic of $x$).\footnote{In fact, $\rho_{\ell}(\Frob_x)$ is semisimple, by \cite[Theorem 7.1]{bakker-shankar-tsimerman:canonicalmodels}}
\item When $p$ is the residue characteristic of the closed point $x$, which lies in the special fiber $\mc{S}_{\kappa(v)}$ for some $v \vert p$, the local system $\rho_p$ is crystalline, has an associated overconvergent $G$-valued $F$-isocrystal $\rho_{p, v}^{\FIso^{\dagger}}$ on $\mc{S}_{\kappa(v)}$, and the linearized crystalline Frobenius of $\rho_{p, v}^{\FIso^{\dagger}}|_{\kappa(x)}$ defines the same class in $[G \git G](\QQ)$ as $\rho_{\ell}(\Frob_x)$ for $\ell \neq p$.
\end{enumerate}
\end{thm} 
We have stated this theorem with the real rank condition on factors of $G^{\mr{ad}}$ to be clear what can be deduced using our previous work \cite{klevdal-patrikis:SVcompatiblearXiv}; but of course, combined with Kisin's work in abelian type, which is needed to handle real rank 1 cases, this yields:
\begin{cor}
Let $(G, X)$ be a Shimura datum such that $Z_G(\QQ)$ is a discrete subgroup of $Z_G(\mathbb{A}_f)$, let $K_0 \subset G(\mathbb{A}_f)$ be a neat compact open subgroup, and let $s \in \mr{Sh}(\CC)$, $S_{K_0, s}$ be as above. Then all the conclusions of Theorem \ref{mainthmintro} hold.
\end{cor}
Finally, we remark that the ``pointwise" compatibility statements in these results are really consequences of the stronger statements that for fixed $v$, all $\rho_{\ell}|_{\mc{S}_{\kappa(v)}}$ (and $\rho_{p, v}^{\FIso^{\dagger}}$) are $G$-companions in the strong sense of \cite{drinfeld:pross}. The proofs will clarify this assertion.

Theorem \ref{mainthmintro} results from combining our previous work \cite{klevdal-patrikis:SVcompatiblearXiv} with the work of Bakker-Shankar-Tsimerman (\cite{bakker-shankar-tsimerman:canonicalmodels}) showing that all Shimura varieties have \textit{canonical} integral models away from a finite set of bad primes (essentially the $N$ of Theorem \ref{mainthmintro}). We should clarify the logical relation between these three papers: \cite{bakker-shankar-tsimerman:canonicalmodels} makes use of Part (1) of Theorem \ref{mainthmintro} (see Lemma 5.1, Theorem 5.2 of \textit{loc. cit.}), which we established for $\rho_{\ell}^{\mr{ad}}$ in \cite[Theorem 2.3]{klevdal-patrikis:SVcompatiblearXiv}. In \S \ref{ramsection}, we give a proof of Part (1) of Theorem \ref{mainthmintro} in general that does not depend on the two earlier papers. Using their main result on integral models along with our main theorem from \cite{klevdal-patrikis:SVcompatiblearXiv} (the analogue of Part (2) of Theorem \ref{mainthmintro} for $\rho_{\ell}^{\mr{ad}}$), Bakker-Shankar Tsimerman prove (\cite[Theorem 9.7]{bakker-shankar-tsimerman:canonicalmodels}) that for $v \nmid N$, every $\mu$-ordinary point of the special fiber $\mc{S}_v$ lifts to a characteristic zero special point.\footnote{To be precise, \cite{bakker-shankar-tsimerman:canonicalmodels} use our earlier result for $\rho_{\ell}^{\mr{ad}}$ to cover the non-abelian type (or the real rank at most 1) cases; in abelian type, \cite[Corollary 2.3.1]{kisin:modpabelian} establishes the desired compatibility.} We in turn use this result on CM lifting of $\mu$-ordinary points to establish Parts (2) and (3) of Theorem \ref{mainthmintro}.

We now summarize each section and the arguments therein. In \S \ref{compsection}, we recall results of Drinfeld (\cite{drinfeld:pross}) on Deligne's companions conjecture; we need his work on companions for representations valued in not necessarily connected reductive (not only semisimple) groups, and so we elaborate slightly on the results in \cite[\S 6]{drinfeld:pross}; we also include a discussion of overconvergent $F$-isocrystal companions. In \S \ref{ramsection}, we show (Corollary \ref{svindram}) that Part (1) of Theorem \ref{mainthmintro} holds: there is an integer $N$ such that for all $\ell$, all $\rho_{\ell}$ are unramified away from $N \ell$. For the $G^{\mr{ad}}$-representations, we showed this in \cite[Proposition 2.3]{klevdal-patrikis:SVcompatiblearXiv} by an abstract argument using superrigidity and Deligne's conjecture. When $G$ is not an adjoint group, descents to $\pi_1(S_{K_0, s})$ of a given representation of $\pi_1^{\mr{top}}(S_{K_0, s}(\CC))$ (extended to $\pi_1((S_{K_0, s})_{\ol{\QQ}})$ using integrality) are not unique, and it is easy to choose descents for varying $\ell$ that have no common finite bad set $N$ as in Part (1) of the Theorem. We need to use the fact that our descents arise from the family of canonical models; it turns out that all we need in addition to the arguments in the adjoint case is that there is a number field point $y$ of $S_{K_0, s}$ for which the specialized local systems $\rho_{\ell}|_y$ are unramified outside $N_y \ell$ for some integer $N_y$ independent of $\ell$. In \S \ref{SVsection} we prove Parts (2) and (3) of Theorem \ref{mainthmintro}. Given the result of \S \ref{ramsection}, from \cite[Theorem 9.2]{bakker-shankar-tsimerman:canonicalmodels} we deduce that on the $\mu$-ordinary locus of the special fiber $\mc{S}_v$ ($v \nmid N \ell \ell'$), $\rho_{\ell}$ and $\rho_{\ell'}$ satisfy the companions property (2). Using the companions construction (\S \ref{compsection}) and Zariski-density of the $\mu$-ordinary locus, we can deduce the companions property on all of $\mc{S}_v$ via a strong multiplicity 1 property for Galois representations. The crystalline compatibility works analogously but requires as additional input the overconvergent ($G$-valued) F-isocrystals associated to the crystalline (by Esnault-Groechenig: \cite[Appendix]{pila-shankar-tsimerman:andre-oort}) local systems $\rho_p$ in our joint work \cite{huryn-kedlaya-klevdal-patrikis}.

We end this introduction by remarking on what we do not do, and what we hope to pursue in future work. We view our results in the framework of the Langlands-Rapoport conjecture, as a significant step toward associating Kottwitz triples to mod $p$ points on Shimura varieties; in abelian type, Kisin has proven the Langlands-Rapoport conjecture (\cite{kisin:modpabelian}), but this remains a wide-open and fundamental problem in non-abelian type. We further note that Kottwitz triples are not only a useful intermediary in work on the full Langlands-Rapoport conjecture (as in \textit{loc. cit.}) but are in principle a sufficient intermediary to calculate the zeta function of the Shimura variety; indeed, this is how Kottwitz first employed them in \cite{kottwitz:modppoints}. Recall (\cite[\S 4.3]{kisin:modpabelian}) that a Kottwitz triple over $k= \mathbb{F}_{p^r}$ is a triple $(\gamma_0, (\gamma_{\ell})_{\ell \neq p}, \delta)$ where
\begin{itemize}
\item $\gamma_0 \in G(\QQ)$, well-defined up to  $G(\ol{\QQ})$-conjugacy;
\item $(\gamma_{\ell})_{\ell \neq p} \in G(\mathbb{A}^p_f)$;
\item $\delta \in G(\mr{Frac}(W(k)))$, well-defined up to Frobenius conjugacy by elements of $G(W(k))$.
\end{itemize}
These data are required to satisfy several conditions:
\begin{itemize}
\item $\gamma_0$ is $G(\ol{\mathbb{A}}^p_f)$-conjugate to $(\gamma_{\ell})_{\ell \neq p}$, where $\ol{\mathbb{A}}_f^p$ is the restricted direct product of the $\ol{\QQ}_{\ell}$ ($\ell \neq p$) with respect to the $\ol{\ZZ}_{\ell}$.
\item $\gamma_0$ is $G(\ol{\QQ}_p)$-conjugate to $\gamma_p:= \delta \sigma(\delta) \cdots \sigma^{r-1}(\delta)$.
\item $\gamma_0 \in G(\RR)$ is elliptic.
\item After possibly replacing $k$ by a finite extension, there is an inner twisting $I$ of $I_0:= \mr{Cent}_G(\gamma_0)$ such that $I \otimes_{\QQ} \RR$ is anisotropic mod center, and for all primes $\ell$, $I \otimes_{\QQ} \QQ_{\ell}$ is isomorphic to $I_{\ell}$ as an inner twist of $I_0 \otimes_{\QQ} \QQ_{\ell}$ (here $I_{\ell}$ for $\ell \neq p$ is $\mr{Cent}_{G_{\QQ_{\ell}}}(\gamma_{\ell})$, and for $\ell = p$ it is the Frobenius-centralizer: see \cite[\S 4.3.1]{kisin:modpabelian} for details).
\end{itemize}
To every $k$-point $x \in \mc{S}(k)$ ($p \nmid N$) we associate a triple $(\gamma_0, (\gamma_{\ell})_{\ell \neq p}, \delta)$, \textit{except} that we have only constructed $\gamma_0$ in general as an element of $[G \git G](\QQ)$. With that caveat, we have verified the first and second desiderata for the triple: $\gamma_{\ell}$ is of course $\rho_{\ell}(\Frob_x)$, and $\delta$ arises from the absolute Frobenius ($\gamma_p$ being the linearized version) on the $F$-isocrystal at $x$ arising from the crystalline local system $\rho_p$; Kisin's integral $p$-adic Hodge theory (\cite[\S 2]{kisin:intmodabelian}) gives the precise well-definedness property of $\delta$. The key question remaining seems to be (for $x$ not in the $\mu$-ordinary locus, where there is no issue) finding the actual rational representative $\gamma \in G(\QQ)$; we hope to return to this problem (and the verification of the remaining conditions) in future work.

\section{Recollection on companions}{\label{compsection}}
For a smooth connected variety $X$ over $\mathbb{F}_p$ 
and a geometric point $\xi$ of $X$, set $\pi_1(X)= \pi_1(X, \xi)$. Drinfeld proves (\cite[Theorem 1.4.1]{drinfeld:pross}) that for any semisimple, not necessarily connected, group $H$ over $\ol{\QQ}$, any prime $\lambda$ not above $p$ of $\ol{\QQ}$, and any continuous representation $\rho_{\lambda} \colon \pi_1(X) \to H(\ol{\QQ}_{\lambda})$ with Zariski-dense image, $\rho_{\lambda}$ has for any other place $\lambda' \nmid p$ a companion
\[
\rho_{\lambda \leadsto \lambda'} \colon \pi_1(X) \to H(\ol{\QQ}_{\lambda'}),
\]
also with Zariski-dense image. This $\rho_{\lambda \leadsto \lambda'}$ has the \textit{weak companion property} that for all closed points $x \in X$, $\rho_{\lambda}(\Frob_x)$ and $\rho_{\lambda \leadsto \lambda'}(\Frob_x)$ define a common element of the GIT quotient (space of conjugacy classes) $[H \git H](\ol{\QQ})$ (inside the $\ol{\QQ}_{\lambda}$ and $\ol{\QQ}_{\lambda'}$ points, respectively); here and throughout we write $\Frob_x$ for a geometric Frobenius element associated to the residue field $\kappa(x)$. It also has a deeper \textit{strong companion property} that characterizes it uniquely even in the absence of the implication ``local Frobenius conjugacy implies global conjugacy." We give Drinfeld's formulation in terms of a ``universal cover" $\wt{X}$ of $X$. Let $\wt{X}$ be a pro-object of the category of finite \'etale covers of $X$ that pro-represents a fiber functor on $\mr{F\acute{E}t}(X)$, let $\Pi= \Aut(\wt{X}/X)$, and choose a geometric point $\tilde{\xi}$ of $\wt{X}$ over our geometric point $\xi$. This choice induces an isomorphism $\Pi \cong \pi_1(X, \xi)^{\mr{op}}$ and (composing with $g \mapsto g^{-1}$) an isomorphism $\Pi \cong \pi_1(X, \xi)$, which will allow the assertions on existence of companions to be transferred to $\pi_1(X)$. 

Now, for fixed $\lambda$, let $\hat{\Pi}_{\lambda}$ be the inverse limit over the targets of all Zariski-dense continuous homomorphisms $\Pi \to H(\ol{\QQ}_{\lambda})$ to semisimple groups $H$. $\Pi$ contains a canonical subset of Frobenius elements $\Pi_{\mr{Frob}}= \{\Frob_{\tilde{x}}^n: \tilde{x} \in |\wt{X}|, n \geq 1\}$ coming from the classical theory of decomposition groups, and our chosen isomorphism transfers this subset to $\pi_1(X)$ as the union of conjugacy classes of positive powers of geometric Frobenii at all closed points. The groups $\hat{\Pi}_{\lambda}$ descend, uniquely up to inner automorphisms by the identity component, to groups $\hat{\Pi}_{(\lambda)}$ over $\ol{\QQ}$. Drinfeld shows that for two places $\lambda$, $\lambda'$, there is a unique isomorphism $\hat{\Pi}_{(\lambda)} \cong \hat{\Pi}_{(\lambda')}$, modulo inner automorphisms by the identity component, preserving the canonical diagrams
\begin{equation}\label{canonical}
\Pi_{\Frob} \to [\hat{\Pi}_{(\lambda)} \git \hat{\Pi}^0_{(\lambda)}](\ol{\QQ}) \to \Pi= \pi_0(\hat{\Pi}_{(\lambda)}). 
\end{equation}
This isomorphism produces the companions $\rho_{\lambda \leadsto \lambda'}$: a homomorphism $\hat{\Pi}_{\lambda} \to H_{\ol{\QQ}_{\lambda}}$ induces a homomorphism $\hat{\Pi}_{(\lambda)} \to H$ (unique up to $\hat{\Pi}_{(\lambda)}^0$-conjugation), which can then be transported to its $\lambda'$-adic version. We will not make explicit use of this stronger companion property, but we will note when (as will be the case) our assertions of compatibility have this stronger meaning. 

In the present paper we will need a variant with reductive monodromy groups; we will also need the companion construction (in the reductive case) when $\lambda' \vert p$. In the case $\lambda, \lambda' \nmid p$, Drinfeld presents one reductive variant in \cite[\S 6]{drinfeld:pross}, where he defines a group $\Pi^{\mr{mot}}$ over $\ol{\QQ}$ that similarly allows for the construction of weak or strong companions for continuous homomorphisms $\rho_{\lambda}  \colon \Pi \to H(\ol{\QQ}_{\lambda})$ with reductive (not necessarily connected) algebraic monodromy group such that in each representation of $H$ and for all $x \in |X|$, the eigenvalues of $\Frob_x$ are all $\# \kappa(x)$-Weil numbers. We will explain this along with a version that produces companions for local systems that are plain of characteristic $p$ (rather than ``motivic," i.e. a direct sum of pure) in the sense of \cite{chin:indl}.

Following \cite[\S 6]{drinfeld:pross}, for $\lambda \nmid p$ let $\hat{\Pi}_{\lambda}^{\mr{red}}$ be the pro-reductive $\lambda$-adic completion of $\Pi$. 
Now let $\lambda \vert p$, and let $\FIso^{\dagger}(X)$ be the category of overconvergent $F$-isocrystals on $X$. Its $\ol{\QQ}_{\lambda}$-linearization $\FIso^{\dagger} \otimes_{\QQ_p} \ol{\QQ}_\lambda$ is a neutral Tannakian category over $\ol{\QQ}_{\lambda}$, and choosing a fiber functor (for instance, by pulling back to a closed point: see \cite[Example 3.1 ff.]{huryn-kedlaya-klevdal-patrikis}) we obtain a Tannakian group $\pi_1^{\FIso^{\dagger}}(X)$. As in the $\lambda \nmid p$ case, we write $\hat{\Pi}_{\lambda}$ and $\Pired$ for the maximal pro-semisimple and maximal pro-reductive quotients of $\pi_1^{\FIso^{\dagger}}(X)$.

We have the following maps:

\begin{itemize}
\item Suppose $\lambda \nmid p$. Then $\Xi_{\lambda} \colon \ol{\ZZ}_{\lambda}^\times \to \Hom(\Pired, \mathbb{G}_m)$ is constructed as the composition
\[
\Xi_{\lambda} \colon \ol{\ZZ}_{\lambda}^\times \simeq \Hom_{\mr{cts}}(\hat{\ZZ}, \ol{\QQ}_{\lambda}^\times) \to \Hom_{\mr{cts}}(\Pi, \ol{\QQ}_{\lambda}^\times)=\Hom(\Pired, \mathbb{G}_m),
\]
where the second map is precomposition with the canonical homomorphism $\Pi \to \hat{\ZZ}$ (under which $\Frob_{\tilde{x}}$ maps to $[\kappa(x):\mathbb{F}_p]$).
\item Suppose $\lambda \vert p$. Then $\Xi_{\lambda} \colon \ol{\QQ}_{\lambda}^\times \to \Hom(\Pired, \mathbb{G}_m)$ is given by interpreting an element of $\ol{\QQ}_{\lambda}^\times$ as a rank-1 object of $\FIso(\Spec(\mathbb{F}_p)) \otimes_{\QQ_p} \ol{\QQ}_{\lambda}$ (see \cite[Example 3.1]{huryn-kedlaya-klevdal-patrikis}) and then pulling back along $X \to \Spec(\mathbb{F}_p)$.
\end{itemize}
These homomorphism (dualized) and the canonical projections $\Pired \to \hat{\Pi}_{\lambda}$ induce the following isomorphisms:
\begin{itemize}
\item Suppose $\lambda \nmid p$. Then 
\[
\Pired \to \hat{\Pi}_{\lambda} \times_{\hat{\ZZ}} \Hom(\ol{\ZZ}_{\lambda}^\times, \mathbb{G}_m)
\]
is an isomorphism by \cite[Proposition 3.5.2]{drinfeld:pross}
\item Suppose $\lambda \vert p$. Then
\[
\Pired \to \hat{\Pi}_{\lambda} \times_{\hat{\ZZ}} \Hom(\ol{\QQ}_{\lambda}^\times, \mathbb{G}_m)
\]
is an isomorphism by Lemma \ref{Pired} below.
\end{itemize}
I owe the following lemma to Jake Huryn's write-up in a draft version of \cite{huryn-kedlaya-klevdal-patrikis}:
\begin{lem}\label{Pired}
Suppose $\lambda \vert p$. Then the map
\[
\Pired \to \hat{\Pi}_{\lambda} \times_{\hat{\ZZ}} \Hom(\ol{\QQ}_{\lambda}^\times, \mathbb{G}_m)
\]
just described is an isomorphism. 
\end{lem}
\begin{proof}
We use the criterion of \cite[Lemma 3.5.1]{drinfeld:pross}. $\ol{\QQ}_{\lambda}^\times/\mu_{\infty}$ is a $\QQ$-vector space, so it suffices to show that for all finite index subgroups $G'$ of $\Pired$, the map
\[
\ol{\QQ}_{\lambda}^\times/\mu_{\infty} \to \Hom(G', \mathbb{G}_m)/\mr{torsion}
\]
induced by $\Xi_{\lambda}$ is an isomorphism. First, by \cite[Proposition B.7.6]{drinfeld-kedlaya:slopes} any finite index subgroup $G' \subset \Pired$ is the image of the injection $\hat{\Pi}_{\lambda, X'}^{\mr{red}} \to \Pired$, where $X'= \wt{X}/U \to X$ is determined by taking the image $U$ of $G'$ in $\Pi$ (and $\hat{\Pi}_{\lambda, X'}^{\mr{red}}$ denotes the analogous pro-reductive completion of $\pi_1^{\FIso^{\dagger}}(X')$). If $\Xi_{\lambda}(\alpha)|_{X'}$ is torsion, then the $F$-isocrystal on $\Spec(\mathbb{F}_p)$ associated to $\alpha$ is already torsion after passing to some finite field extension $k/\mathbb{F}_p$ (e.g., the field of constants of $X'$), where the linearized Frobenius acts as $\alpha^{[k: \mathbb{F}_p]}$. Thus $\alpha \in \mu_{\infty}$ is torsion. 

For surjectivity, we use the result of Crew ($\dim(X)=1$) and Abe (the general case) that up to a twist pulled back from the ground field any rank 1 overconvergent $F$-isocrystal on a finite-type $\mathbb{F}_p$-scheme is torsion (\cite[Lemma 6.1]{abe:langlandsp}). Let $\mc{E}$ be the rank 1 overconvergent $F$-isocrystal on $X'$ corresponding to an element of $\Hom(\hat{\Pi}_{\lambda, X'}^{\mr{red}}, \mathbb{G}_m)$; then there exists a rank 1 $F$-isocrystal $\mc{F}$ on $X'$ that is pulled back from $\Spec(\mathbb{F}_p)$ such that $\mc{E} \otimes \mc{F}$ is torsion. Thus, up to twisting by torsion objects $\mc{E}$ lies in the image of $\Xi_{\lambda}$, as desired. 

\end{proof}

Replacing $\ol{\ZZ}_{\lambda}^\times$ by its subgroup $\mc{W}_p \subset \ol{\QQ}^\times \cap \ol{\ZZ}_{\lambda}^\times$ (intersection inside $\ol{\QQ}_{\lambda}^\times$) of all $p$-Weil numbers, Drinfeld obtains a pro-reductive group scheme quotient $\hat{\Pi}_{\lambda}^{\mr{red}} \twoheadrightarrow \hat{\Pi}_{\lambda}^{\mr{mot}}$ whose finite-dimensional representations are those of $\hat{\Pi}_{\lambda}^{\mr{red}}$ such that the corresponding $\lambda$-adic representation has, for all $x \in |X|$, $\Frob_x$-eigenvalues that are $\#\kappa(x)$-Weil numbers. Exactly the same construction applies if instead of $\mc{W}_p$ we take the subgroup
\[
\mc{P}_p \subset \ol{\QQ}^\times \cap \ol{\ZZ}_{\lambda}^\times
\]
of algebraic numbers that are plain of characteristic $p$ (units at all finite places away from $p$), yielding an intermediate quotient
\[
\hat{\Pi}_{\lambda}^{\mr{red}} \twoheadrightarrow \hat{\Pi}_{\lambda}^{\mr{plain}} \twoheadrightarrow \hat{\Pi}_{\lambda}^{\mr{mot}}
\]
that we claim classifies continuous semisimple $\lambda$-adic representations whose Frobenius eigenvalues are all plain of characteristic $p$. The analogous claim is made in \cite[\S 6]{drinfeld:pross} with $\mc{W}_p$ in place of $\mc{P}_p$, and in covering both cases and the case of $p$-adic coefficients we will take the opportunity to give a few more details than \textit{loc. cit.}

For any $\lambda$, define $\hat{\Pi}_{\lambda}^{\mr{plain}}$ and $\hat{\Pi}_{\lambda}^{\mr{mot}}$ to be the quotients of $\hat{\Pi}_{\lambda}^{\mr{red}}$ such that:
\begin{itemize}
\item If $\lambda \nmid p$, a homomorphism $\rho \colon \hat{\Pi}_{\lambda}^{\mr{red}} \to H$ factors through $\hat{\Pi}_{\lambda}^{\mr{plain}}$ (respectively, $\Pimot$) if and only if for all representations $r \colon H \to \mr{GL}_N$, the $\lambda$-adic representation $r \circ \rho_{\lambda}$ associated to $r \circ \rho$ satisfies: for all closed points $x \in X$, all eigenvalues of $r \circ \rho_{\lambda}(\Frob_x)$ are plain of characteristic $p$ (respectively, are $\#\kappa(x)$-Weil numbers).
\item If $\lambda \vert p$, a homomorphism $\rho \colon \hat{\Pi}_{\lambda}^{\mr{red}} \to H$ factors through $\hat{\Pi}_{\lambda}^{\mr{plain}}$ (respectively, $\Pimot$) if and only if for all representations $r \colon H \to \mr{GL}_N$, the object $\mc{F}$ of $\FIso^{\dagger}(X)$ associated to $r \circ \rho$ satisfies: for all closed points $x \in X$, the eigenvalues of the $\kappa(x)$-linearized Frobenius on $x^* \mc{F}$ are plain of characteristic $p$ (respectively, are $\# \kappa(x)$-Weil numbers).
\end{itemize}
(These constructions make sense, since for any $H \to H'$ and any $\rho \colon \Pired \to H$ with one of the properties above, the composition $\Pired \xrightarrow{\rho} H \to H'$ also has the same property: we can then define $\Pimot$ and $\Piplain$ as projective limits over surjective homomorphisms with the required properties.) We will deduce the following lemma from \cite[Lemma 3.5.1, Proposition 3.5.2]{drinfeld:pross} and Lemma \ref{Pired}: 
\begin{lem}\label{piplainlem}
For any finite place $\lambda$ of $\ol{\QQ}$, there are canonical isomorphisms
\[
\Piplain \to \hat{\Pi}_{\lambda} \times_{\hat{\ZZ}} \Hom(\mc{P}_p, \mathbb{G}_m)
\]
and 
\[
\Pimot \to \hat{\Pi}_{\lambda} \times_{\hat{\ZZ}} \Hom(\mc{W}_p, \mathbb{G}_m).
\]
\end{lem}
\begin{proof}
In each case, the first factor of the map is the projection to the pro-semisimple quotient. To define the second factor, we construct the dual map $A \to \Hom(G, \mathbb{G}_m)$ for $(A, G)$ equal to $(\mc{P}_p, \Piplain)$ or $(\mc{W}_p, \Pimot)$ by restricting the source of the maps $\Xi_{\lambda}$.
\begin{itemize}
\item Suppose $\lambda \nmid p$. For $\alpha \in \mc{P}_p$ (respectively, $\alpha \in \mc{W}_p$), viewing $\Xi_{\lambda}(\alpha)$ as a homomorphism $\Pi \to \ol{\QQ}_{\lambda}^\times$, by construction $\Xi_{\lambda}(\alpha)(\Frob_x)= \alpha^{\#\kappa(x)}$; thus $\Xi_{\lambda}(\alpha)$ factors through $\Piplain$ (respectively $\Pimot$). 
\item Suppose $\lambda \vert p$. For $\alpha \in \mc{P}_p$ (respectively, $\alpha \in \mc{W}_p$), $\Xi_{\lambda}(\alpha)$, viewed as a rank 1 object of $\FIso(\Spec(\mathbb{F}_p)) \otimes_{\QQ_p} \ol{\QQ}_{\lambda}$, has the property that for all closed points $x$, $x^*(\Xi_{\lambda}(\alpha))$ is the rank-1 object of $\FIso(\Spec(\kappa(x))) \otimes_{\QQ_p} \ol{\QQ}_{\lambda}$ whose $\#\kappa(x)$-linearized Frobenius has eigenvalue $\alpha^{\# \kappa(x)}$; thus $\Xi_{\lambda}(\alpha)$ factors through $\Piplain$ (respectively, $\Pimot$).
\end{itemize}
Let $Z^{\mr{red}, 0}$, $Z^{\mr{plain}, 0}$, and $Z^{\mr{mot}, 0}$ be the connected components of the identity of the centers of $\Pired$, $\Piplain$, and $\Pimot$. For any pro-reductive group $G$, we have a short exact sequence
\[
1 \to Z_G^0 \to G \to G^{\mr{ss}} \to 1
\]
with $G^{\mr{ss}}$ the maximal pro-semisimple quotient. In all three cases under consideration, $G^{\mr{ss}}= \hat{\Pi}_{\lambda}$, so we see that 
\[
\ker(\Pired \to \Pimot)= \ker(Z^{\mr{red}, 0} \to Z^{\mr{mot}, 0})
\]
and
\[
\ker(\Pired \to \Piplain)= \ker(Z^{\mr{red}, 0} \to Z^{\mr{plain}, 0}).
\]
From the first part of the argument of \cite[Lemma 3.5.1]{drinfeld:pross}, we want to show that for each possibility $G= \Pired$, $G=\Piplain$, $G= \Pimot$ and (respectively) $A= \ol{\ZZ}_{\lambda}^\times$ ($\lambda \nmid p$) or $\ol{\QQ}_{\lambda}^\times$ ($\lambda \vert p$), $A= \mc{P}_p$, $A= \mc{W}_p$, the two maps
\[
A/A_{\mr{tors}} \xrightarrow{f} \Hom(G^0, \mathbb{G}_m) \xrightarrow{g} \Hom(Z_G^0, \mathbb{G}_m)
\]
are both isomorphisms. For $G= \Pired$ and $\lambda \nmid p$, this is \cite[Proposition 3.5.2]{drinfeld:pross}, and for $G= \Pired$ and $\lambda \vert p$, this is Lemma \ref{Pired}. We now explain how the other cases follow formally. First note that the map $f$ is an isomorphism in all three cases. To see this, we observe from the definitions that
\[
\xymatrix{
\ol{\QQ}_{\lambda}^\times \ar[r] & \Hom (\Pired, \mathbb{G}_m) \\
\mc{P}_p \ar[r] \ar@{^{(}->}[u] & \Hom(\Piplain, \mathbb{G}_m) \ar@{^{(}->}[u]
}
\]
is for $\lambda \vert p$ a pullback square and is moreover modulo torsion a pullback square; the same holds for $\lambda \nmid p$ with $\ol{\ZZ}_{\lambda}^\times$ in place of $\ol{\QQ}_{\lambda}^\times$, and with $\mc{W}_p$ and $\Pimot$ replacing $\mc{P}_p$ and $\Piplain$. Thus $\mc{P}_p/\mu_{\infty} \xrightarrow{\sim} \Hom(\Piplain, \mathbb{G}_m)/\Hom(\Piplain, \mathbb{G}_m)_{\mr{tors}}$ for all $\lambda$, and likewise for $\Pimot$. The same holds with $X$ replaced by any finite \'etale connected cover, so $\mc{P}_p/\mu_{\infty} \xrightarrow{\sim} \Hom((\Piplain)^0, \mathbb{G}_m)$ is also an isomorphism (again, likewise for $\Pimot$).

Injectivity of the map $g$ holds for any reductive group. For surjectivity, let $G$ be one of $\Pimot$ or $\Piplain$. Any element of $\Hom(Z_G^0, \mathbb{G}_m) \subset \Hom(Z^{\mr{red}, 0}, \mathbb{G}_m)$ is (by Lemma \ref{Pired} and \cite[Proposition 3.5.2]{drinfeld:pross}) the restriction to $Z^{\mr{red}, 0}$ of $\Xi_{\lambda}(\alpha)$ for some $\alpha \in \ol{\ZZ}_{\lambda}^\times$ ($\lambda \nmid p$) or $\alpha \in \ol{\QQ}_{\lambda}^\times$ ($\lambda \vert p$). Since $\ker(\Pired \to G)= \ker (Z^{\mr{red}, 0} \to Z_G^0)$, $\Xi_{\lambda}(\alpha) \colon \Pired \to \mathbb{G}_m$ must factor through $G \to \mathbb{G}_m$; this (as above) forces $\alpha$ to lie in $\mc{W}_p$ or $\mc{P}_p$ according to the choice of $G$. We conclude that $g \circ f$ is surjective, hence $g \circ f$ is an isomorphism (as are both $g$ and $f$), hence
\[
\Piplain \to \hat{\Pi}_{\lambda} \times_{\hat{\ZZ}} \Hom(\mc{P}_p, \mathbb{G}_m)
\]
and 
\[
\Pimot \to \hat{\Pi}_{\lambda} \times_{\hat{\ZZ}} \Hom(\mc{W}_p, \mathbb{G}_m)
\]
are both isomorphisms.

\end{proof}
\begin{rmk}
Ultimately these results (and \cite[Proposition 3.5.2]{drinfeld:pross} rely on class field theory and its bootstrapping to the semisimplicity of geometric monodromy.
\begin{enumerate}
\item Concretely, for $\lambda \nmid p$ the proof that a homomorphism $\rho_{\lambda} \colon \hat{\Pi}_{\lambda} \times_{\hat{\ZZ}} \Hom(\mc{P}_p, \mathbb{G}_m) \twoheadrightarrow H$ classifies a $\lambda$-adic representation that is plain of characteristic $p$ (and similarly for the ``weakly motivic" $\mc{W}_p$ case) can be interpreted as follows: the plain (or weakly motivic) property can be checked after restriction to a finite-index subgroup of $\Pi$, and from the semisimplicity of geometric monodromy one can check that after replacing $X$ by a suitable finite \'etale connected cover, $\rho_{\lambda}$ can be written as a product of homomorphisms $\hat{\Pi}_{\lambda} \to H^{\mr{geo}, 0}$ and $\Hom(\mc{P}_p, \mathbb{G}_m) \to Z_{H^0}^0$, with $H^{\mr{geo}}$ denoting the geometric monodromy group. The former is by Lafforgue's theorem pure weight zero in any finite-dimensional representation, and the latter corresponds (in a choice of basis of $Z_{H^0}^0 \cong \mathbb{G}_m^r$) to an $r$-tuple of elements of $\mc{P}_p$ that make the plainness conclusion evident. The converse (that plain representations of $\Pi$ correspond to homomorphisms out of $\Pired$ that factor through $\hat{\Pi}_{\lambda} \times_{\hat{\ZZ}} \Hom(\mc{P}_p, \mathbb{G}_m)$) can be argued similarly.
\item To show a homomorphism $\Pi \to H(\ol{\QQ}_{\lambda})$ is plain (respectively, weakly motivic) of characteristic $p$, it suffices to know $\rho_{\lambda}(\Frob_x)$ has plain (or $\# \kappa(x)$-Weil numbers, of possibly different weights) in a faithful representation for a single closed point $x$: compare \cite[Proposition 4.3]{chin:indl}. This can be seen from the sketch in item (1), or in terms of our proof by considering the restriction of $\rho_{\lambda}$ to the subgroup $1 \times \Hom(\mc{P}_p/\mu_{\infty}, \mathbb{G}_m) \xrightarrow{\rho_{\lambda}} Z_{H^0}^0$: up to roots of unity, this homomorphism is determined by what it does on a single Frobenius element.
\item Since $\mc{P}_p/\mu_{\infty}$ is a $\QQ$-vector space, the ``plain" case of the Lemma can be deduced directly from the criterion of \cite[Lemma 3.5.1]{drinfeld:pross}, as Drinfeld does for $\Pired$. Note that $\mc{W}_p/\mu_{\infty}$ is not a $\QQ$-vector space, so in this case arguing as in Lemma \ref{piplainlem} seems to be necessary.
\end{enumerate}
\end{rmk}

Drinfeld's main theorem (\cite[Theorem 1.4.1]{drinfeld:pross}, and see \cite[Theorem 3.8]{huryn-kedlaya-klevdal-patrikis} for the details when $\lambda \vert p$) gives for all $\lambda$ and $\lambda'$ a unique isomorphism $\hat{\Pi}_{(\lambda)} \simeq \hat{\Pi}_{(\lambda')}$ in $\Pross(\ol{\QQ})$ (the category of pro-semisimple groups, with homomorphisms taken up to conjugation by the identity component) preserving the canonical diagrams of Equation \eqref{canonical}. He thus obtains an object $\hat{\Pi}$ of $\Pross(\ol{\QQ})$ that is unique up to a unique isomorphism. 
\begin{cor}\label{plainindl}
For all $\lambda$ and $\lambda'$ there is a unique isomorphism $\hat{\Pi}_{(\lambda)}^{\mr{plain}} \xrightarrow{\sim} \hat{\Pi}_{(\lambda')}^{\mr{plain}}$ sending the canonical diagram
\begin{equation}\label{canonicalplain}
\Pi_{\Frob} \to [\hat{\Pi}_{(\lambda)}^{\mr{plain}} \git \hat{\Pi}_{(\lambda)}^{\mr{plain}, 0}](\ol{\QQ}) \to \Pi
\end{equation}
to the corresponding diagram with $\lambda'$ in place of $\lambda$. Similarly, there is a unique such isomorphism $\hat{\Pi}_{(\lambda)}^{\mr{mot}} \xrightarrow{\sim} \hat{\Pi}_{(\lambda')}^{\mr{mot}}$.
\end{cor}
\begin{proof}
Under the isomorphism of Lemma \ref{piplainlem}, the isomorphism $\hat{\Pi}^{\mr{plain}}_{(\lambda)} \xrightarrow{\sim} \hat{\Pi}^{\mr{plain}}_{(\lambda')}$ is induced by the unique isomorphism $\hat{\Pi}_{(\lambda)} \xrightarrow{\sim} \hat{\Pi}_{(\lambda')}$ preserving Equation \eqref{canonical} and the identity on $\Hom(\mc{P}_p, \mathbb{G}_m)$. For the uniqueness, it suffices (by the \v{C}ebotarev density theorem and the uniqueness in \cite[Theorem 1.4.1]{drinfeld:pross}) to check that an (on the nose) automorphism $\alpha$ of $\hat{\Pi}^{\mr{plain}}_{\lambda}$ that is trivial on $[\Piplain \git \hat{\Pi}_{\lambda}^{\mr{plain}, 0}]$ and induces the identity on the pro-semisimple quotient $\hat{\Pi}_{\lambda}$ is the identity. Since $\alpha$ induces the identity on $\hat{\Pi}_{\lambda}$, it has the form $\alpha= \mr{id} \cdot \chi$ for some homomorphism $\chi \colon \Piplain \to Z^{\mr{plain}, 0} \cong \Hom(\mc{P}_p/\mu_{\infty}, \mathbb{G}_m)$. Since 
\[
[\Piplain \git \hat{\Pi}_{\lambda}^{\mr{plain}, 0}] \cong [\hat{\Pi}_{\lambda} \git \hat{\Pi}_{\lambda}^0] \times_{\hat{\ZZ}} \Hom(\mc{P}_p, \mathbb{G}_m),
\]
and $Z^{\mr{plain}, 0}$ injects into the right hand side (as $\ker(\Hom(\mc{P}_p, \mathbb{G}_m) \to \hat{\ZZ})$), we must have $\chi=1$, as needed.
\end{proof}
As in \cite[\S 6]{drinfeld:pross}, we can then define groups over $\ol{\QQ}$, $\hat{\Pi}^{\mr{plain}}= \hat{\Pi} \times_{\hat{\ZZ}} \Hom(\mc{P}_p, \mathbb{G}_m)$ and $\hat{\Pi}^{\mr{mot}}= \hat{\Pi} \times_{\hat{\ZZ}} \Hom(\mc{W}_p, \mathbb{G}_m)$: these are objects of $\Prored(\ol{\QQ})$ defined up to unique isomorphism, and for all $\lambda$, $\hat{\Pi}^{\mr{plain}} \times_{\ol{\QQ}} \ol{\QQ}_{\lambda}$ is isomorphic, uniquely up to conjugation by $\hat{\Pi}_{\lambda}^{\mr{plain}, 0}$, to $\hat{\Pi}_{\lambda}^{\mr{plain}}$ (and likewise for $\hat{\Pi}^{\mr{mot}}$).

We will apply the above conclusions in the following simplified form:
\begin{cor}\label{companions}
Let $\rho_{\lambda} \colon \Piplain \to H_{\ol{\QQ}_{\lambda}}$ be a surjection onto a reductive (not necessarily connected) group $H/ \ol{\QQ}$. 
Then $\rho_{\lambda}$ has a strong $\lambda'$-companion $\rho_{\lambda \leadsto \lambda'}$, also plain, for any finite place $\lambda'$ of $\ol{\QQ}$. In particular:
\begin{itemize} 
\item $H$ is the image of the homomorphism $\hat{\Pi}^{\mr{plain}}_{\lambda'} \to H_{\ol{\QQ}_{\lambda'}}$ classifying $\rho_{\lambda \leadsto \lambda'}$ 
\item For any closed point $x \in X$, the associated conjugacy class $\rho_{\lambda}(\Frob_x) \in [H \git H](\ol{\QQ}_{\lambda})$ is $\ol{\QQ}$-rational and equal to the associated conjugacy class $\rho_{\lambda \leadsto \lambda'}(\Frob_x)$.
\end{itemize}
Here we write $\rho_{\lambda}(\Frob_x)$ for the conjugacy class of the $\kappa(x)$-linearized Frobenius associated to the closed point, whether we are in the isocrystal or the lisse sheaf situation.
\end{cor}
\begin{proof}
This is immediate from Corollary \ref{plainindl}: descending $\rho_{\lambda}$ to $\ol{\QQ}$, we get $\hat{\Pi}^{\mr{plain}} \twoheadrightarrow H$ (unique up to $H^0$-conjugation), which in turn for any $\lambda'$ extends to the homomorphism $\hat{\Pi}_{\lambda'}^{\mr{plain}} \twoheadrightarrow H_{\ol{\QQ}_{\lambda'}}$ classifying the companion $\rho_{\lambda \leadsto \lambda'}$. 
\end{proof}
\begin{rmk}
Note we do not claim the conjugacy class $\rho_{\lambda}(\Frob_x)$ is independent of $\lambda$ as an element of $[H \git H^0](\ol{\QQ})$ (compare the assertion in Corollary \ref{plainindl}: indeed, this does not even make sense when only the closed point $x \in X$ rather than its lift $\tilde{x} \in \wt{X}$ is specified. The coarser nature of Corollary \ref{companions} allows us to transfer its conclusions to $\pi_1(X)$ without any ambiguity resulting from the non-canonical isomorphism $\pi_1(X) \cong \Pi$.
\end{rmk}

\section{Ramification}\label{ramsection}

In \cite[Proposition 2.3]{klevdal-patrikis:SVcompatiblearXiv}, we showed that Zariski-dense $G^{\mr{ad}}(\CC)$-valued representations $\rho$ of $\pi_1^{\mr{top}}(X(\CC))$ in the ``superrigid" (\textit{loc. cit.} Definition 2.2) regime have arithmetic descents with independent of $\ell$ ramification (the same would work for Zariski-dense cohomologically rigid representations with quasi-unipotent local monodromy and finite-order abelianization). Without Zariski-density, and  more precisely when the centralizer of $\rho$ is non-trivial, the argument breaks down, and indeed there is no reason that arbitrarily chosen arithmetic descents $\rho_{\lambda} \colon \pi_1(X_F) \to G(\ol{\QQ}_{\lambda})$ should have this property, since one can for varying $\lambda$ twist $\rho_{\lambda}$ by homomorphisms $\Gal(\ol{\QQ}/F) \to Z_G(\rho_{\lambda})(\ol{\QQ}_{\lambda})$ with different ramification sets (the basic case to keep in mind, even with $\rho^{\mr{ad}}$ Zariski-dense, is when this centralizer is the center of a non-adjoint group $G$). In this section (Lemma \ref{indram}), we show that having one number field specialization with independent of $\ell$ ramification is enough to show that a collection of descents also has this property.

Let $G$ be a connected reductive group; since we will work with sufficiently large coefficients, we assume $G$ is split over $\ZZ$. In this section we work with a smooth quasi-projective variety $X$ over a number field $F$ whose topological fundamental group $\Gamma= \pi_1^{\mr{top}}(X(\CC), \bar{x})$ (we will specify a base-point later) is $G^{\mr{ad}}$-superrigid in the sense of \cite[Definition 2.2]{klevdal-patrikis:SVcompatiblearXiv}: that is, for any algebraically closed field $\Omega$ of characteristic zero and any two Zariski-dense homomorphisms $\rho_1, \rho_2 \colon \Gamma \to G^{\mr{ad}}(\Omega)$, there is a $\tau \in \Aut(G^{\mr{ad}}_{\Omega})$ such that $\tau(\rho_1)=\rho_2$.\footnote{Alternatively, throughout this section we could take cohomologically rigid representations with bounded order of quasi-unipotency, without assuming superrigidity properties of $\Gamma$; somewhat longer versions of the arguments then apply using the technique of \cite{esnault-groechenig:rigid}, \cite{klevdal-patrikis:G-rigid}.} We fix once and for all a smooth projective compactification $X \subset \ol{X}$ with $\ol{X} \setminus X$ a strict normal crossings divisor. We further let $N$ be a large enough integer that the setup $(X, \ol{X}, \ol{X} \setminus X)$ spreads out to a relative strict normal crossings divisor over $\mc{O}_{F}[1/N]$, and we choose some $x \in X(\mc{O}_{F'}[1/N])$ with $F'/F$ a finite extension.

For a fixed finite subgroup $Z \subset Z_G$, which we may assume contains $Z_{G^{\mr{der}}}$, let $\mc{R}_Z$ be the set of $G(\CC)$-conjugacy classes of homomorphisms $\rho \colon \Gamma \to G^{\mr{der}}Z(\CC)$ such that $\rho^{\mr{ad}}$ is Zariski-dense in $G^{\mr{ad}}$ (fixing $Z$ amounts to bounding the abelianization). Set $G_Z= G^{\mr{der}} \cdot Z$. Since $\Gamma$ is $G^{\mr{ad}}$-superrigid, $\mc{R}_Z$ is finite. As in \cite{esnault-groechenig:rigid}, \cite{klevdal-patrikis:G-rigid}, all members of $\mc{R}_Z$ are integral, so we may choose representatives such that each $\rho \in \mc{R}_Z$ factors through $G(\mc{O}_L)$ for some number field $L \subset \ol{\QQ}$. In turn, for each place $\lambda$ of $\ol{\QQ}$ we obtain the representations 
\[
\rho_{\lambda, \ol{\QQ}} \colon \pi_1(X_{\ol{\QQ}}, \bar{x}) \to G_Z(\ol{\QQ}_{\lambda}).
\]
\begin{lem}\label{tame}
There is an integer $N'$ such that for every $\rho \in \mc{R}_Z$, for every $\lambda$, and for every place $\bar{v}$ of $\ol{\QQ}$ not dividing $N'\ell$, $\rho_{\lambda, \ol{\QQ}}$ factors through the tame specialization map $\pi_1(X_{\ol{\QQ}}, \bar{x}) \xrightarrow{\mr{sp}} \pi_1^t(X_{\kappa(\bar{v})}, \bar{x}_v)$.
\end{lem}
\begin{proof}
For each $\lambda$, there is a finite extension $F(\lambda)/F$ such that for all $\rho \in \mc{R}_Z$, $\rho_{\lambda, \ol{\QQ}}$ descends to a continuous homomorphism
\[
\rho_{\lambda, Z} \colon \pi_1(X_{F(\lambda)}, \bar{x}) \to G_Z(\ol{\QQ}_{\lambda}).\footnote{This may be a bad descent for arithmetic purposes, e.g. not de Rham, but for the present geometric argument it suffices.}
\]
Indeed, this follows from a standard argument (variant of \cite[Theorem 4]{simpson:higgs}, \cite{esnault-groechenig:rigid}, \cite{klevdal-patrikis:G-rigid}): the outer action of $\Gal(\ol{\QQ}/F)$ on $\pi_1(X_{\ol{\QQ}}, \bar{x})$ induces a discrete action on the finite set $\mc{R}_Z$, so there is an open subgroup $H \subset \Gal(\ol{\QQ}/F)$ such that $H$ fixes each $\rho_{\lambda, \ol{\QQ}}$ (up to isomorphism). The adjoint descent $\rho_{\lambda}^{\mr{ad}}$ exists uniquely (using Zariski-density: see \cite[Lemma 2.1]{klevdal-patrikis:SVcompatiblearXiv}) over $F(\lambda)^{\mr{ad}}:= \ol{\QQ}^{H}$. As we vary $\rho$, the $\rho_{\lambda}^{\mr{ad}}$ all factor through $G^{\mr{ad}}(\mc{O})$ for $\mc{O}$ the ring of integers in a finite extension of $\QQ_{\ell}$, and $G_Z(\mc{O}) \to G^{\mr{ad}}(\mc{O})$ is on an open subgroup of the source an isomorphism onto its image (an open subgroup of the target). Therefore after restricting to a finite extension $F(\lambda)/F(\lambda)^{\mr{ad}}$, each $\rho_{\lambda}^{\mr{ad}}$ lifts to
\[
\rho_{\lambda} \colon \pi_1(X_{F(\lambda)}, \bar{x}) \to G_Z(\mc{O}).
\]
Each $\rho_{\lambda}$ factors through $\pi_1(X_{\mc{O}_{F(\lambda)}[1/N_{\lambda}]}, \bar{x})$ for an integer $N_{\lambda}$ (which we may assume divisible by $\ell$) that depends on $\lambda$ but not on $\rho \in \mc{R}_Z$. As in \cite[Proposition 2.3]{klevdal-patrikis:SVcompatiblearXiv}, set $N'= \mr{gcd}(N_{\lambda})$, and suppose that $p$ does not divide $N'$. We claim that for all $\lambda' \nmid p$, $\rho_{\lambda', \ol{\QQ}}$ factors via the tame specialization map through $\pi_1^t(X_{\ol{\kappa(v)}}, \bar{x}_v)$ for every $v \vert p$. Suppose that $p$ divides $N_{\lambda'}$ (else the assertion is clear). For some $\lambda \neq \lambda'$, $p$ does not divide $N_{\lambda}$. Let $\bar{v}$ be any place of $\ol{\QQ}$ above $p$, and let $v= \bar{v}|_{F(\lambda)}$, so for each $\rho \in \mc{R}_Z$ we have the (tame: see \cite[Footnote 3]{klevdal-patrikis:SVcompatiblearXiv}) restriction to the special fiber
\[
\rho_{\lambda, v} \colon \pi_1(X_{\kappa(v)}, \bar{x}_v) \to G_Z(\ol{\QQ}_{\lambda}).
\]
By \cite[Theorem 1.4.1]{drinfeld:pross}, $\rho_{\lambda, v}$ has a (strong) $\lambda'$-companion 
\[
\rho_{\lambda \leadsto \lambda', v} \colon \pi_1(X_{\kappa(v)}, \bar{x}_v) \to G_Z(\ol{\QQ}_{\lambda'})
\]
with the ``same" algebraic monodromy group as $\rho_{\lambda, v}$. (We don't assert this is all of $G_Z$; it is some group between $G^{\mr{der}}$ and $G_Z$.) Since $\rho_{\lambda \leadsto \lambda', v}$ is also tamely ramified, we can as in \cite{esnault-groechenig:rigid}, \cite{klevdal-patrikis:G-rigid} pull back along the tame specialization map to a homomorphism
\[
\rho_{\lambda \leadsto \lambda', \bar{v}}^{\mr{top}} \colon \pi_1^{\mr{top}}(X(\CC), \bar{x}) \to G_Z(\ol{\QQ}_{\lambda'}),
\]
which clearly lies in $\mc{R}_Z$.\footnote{Had we not assumed superrigidity and instead assumed the elements of $\mc{R}_Z$ to be cohomologically rigid, this conclusion would still be valid but would require the arguments of \cite{esnault-groechenig:rigid} and \cite{klevdal-patrikis:G-rigid}.} We may view the elements of $\mc{R}_Z$ as $\ol{\QQ}_{\lambda'}$-valued using that they all have representatives factoring through $G(\mc{O}_L) \subset G(\ol{\QQ}_{\lambda'})$, and we claim that the map $\mc{R}_Z \to \mc{R}_Z$, $\rho \mapsto \rho_{\lambda \leadsto \lambda', \bar{v}}^{\mr{top}}$, is a bijection. Indeed, let $\rho$ and $\sigma$ be elements of $\mc{R}_Z$ such that $\rho_{\lambda \leadsto \lambda', \bar{v}}^{\mr{top}} = \sigma_{\lambda \leadsto \lambda', \bar{v}}^{\mr{top}}$ (equality up to $G$-conjugacy). Then $\rho_{\lambda \leadsto \lambda', \bar{v}}= \sigma_{\lambda \leadsto \lambda' \bar{v}}$. Let $F_v \in \pi_1(X_{\kappa(v)}, \bar{x}_v)$ be the $v$-Frobenius induced by the section ($\kappa(v)$-rational point) $x_v$. The equality $\rho_{\lambda \leadsto \lambda', \bar{v}}(F_v \gamma F_v^{-1})= \sigma_{\lambda \leadsto \lambda', \bar{v}}(F_v \gamma F_v^{-1})$, for all $\gamma \in \pi_1(X_{\ol{\kappa(v)}}, \bar{x}_v)$ implies that $\sigma_{\lambda \leadsto \lambda', v}(F_v)^{-1} \rho_{\lambda \leadsto \lambda', v}(F_v)$ centralizes the image of $\rho_{\lambda \leadsto \lambda', \bar{v}}=\sigma_{\lambda \leadsto \lambda', \bar{v}}$, hence lies in $Z$. As $Z$ is finite and central, passing to a finite extension $\kappa'/\kappa(v)$, we see $\sigma_{\lambda \leadsto \lambda', v}|_{X_{\kappa'}}=\rho_{\lambda \leadsto \lambda', v}|_{X_{\kappa'}}$. Applying \cite[Theorem 1.4.1]{drinfeld:pross} again shows that $\sigma_{\lambda, v}|_{X_{\kappa'}}= \rho_{\lambda, v}|_{X_{\kappa'}}$, so $\sigma_{\lambda, \bar{v}}= \rho_{\lambda, \bar{v}}$, and finally $\sigma= \rho$.

\end{proof}

\begin{lem}\label{indram}
Assume that for some $\rho \in \mc{R}_Z$, there are for all $\lambda$ descents $\rho_{\lambda}$ as above such that
\begin{enumerate}
\item The fields $F(\lambda)$ are independent of $\lambda$.
\item For our given rational point $x \in X(\mc{O}_{F'}[1/N])$, there is an integer $N_x$ (divisible by $N$) such that the specialization $\rho_{\lambda, x}$ is unramified outside $N_x \ell$ for all places $\lambda$. 
\end{enumerate}
Then there is an integer $N''$ such that for every $\lambda$, $\rho_{\lambda}$ factors through $\pi_1(X_{\mc{O}_F[1/N'' \ell]}, \bar{x})$.
\end{lem}
\begin{proof}
For notational simplicity we assume $F(\lambda)=F$ for all $\lambda$ (if this entails enlarging $F$, we can correspondingly enlarge $F'$). The point $x$ induces a splitting 
\[
\pi_1(X_{F'}, \bar{x}) \simeq \pi_1(X_{\ol{\QQ}}, \bar{x}) \rtimes \Gal(\ol{\QQ}/F').
\]
Let $N''$ be the product of the integer $N'$ from Lemma \ref{tame}, the primes (below those places) ramified in $F'/F$, and the integer $N_x$. Let $v \vert p$ be a place of $F$ such that $v$ does not divide $N'' \ell$. Let $w$ be a place of $F'$ above $v$. Let $\mc{O}^{\mr{h}}$ be the henselization of the local ring $\mc{O}_{F', {(w)}}$, let $F^{\mr{h}}$ be $\mr{Frac}(\mc{O}^{\mr{h}})$, let $\mc{O}^{\mr{sh}}$ the strict henselization of $\mc{O}_{F', {(w)}}$, and let $F^{\mr{sh}}= \mr{Frac}(\mc{O}^{\mr{sh}})$, all constructed as subrings of $\ol{\QQ}$.\footnote{We hope the notation $\mc{O}^{\mr{h}}$, etc., causes no confusion, since it does not refer to $F'$ and $w$; it instead has the advantage of being uncluttered. I think one should actually work over $\mc{O}$ the henselization of the local ring $\mc{O}_{F', (w)}$. The analysis below applies there, and this connects more easily to the global theory: to see a finite map with normal source $X' \to X_{\mc{O}_F'}[1/N]$ is \'etale above .} We will show that $\rho_{\lambda}|_{X_{F^{\mr{h}}}}$ factors through $\pi_1(X_{\mc{O}^{\mr{h}}})$. Restricting the splitting $x$ to
\[
\pi_1(X_{F^{\mr{h}}}, \bar{x}) \simeq \pi_1(X_{\ol{\QQ}}, \bar{x}) \rtimes \Gal(\ol{F}/F^{\mr{h}}),
\]
we obtain a factorization
\begin{equation}
\xymatrix{
\pi_1(X_{F^{\mr{h}}}, \bar{x}) \ar@/^2pc/[rr]^{\rho_{\lambda}} \ar@{->>}[r] & \pi_1(X_{\mc{O}^{\mr{sh}}}, \bar{x}) \rtimes \pi_1(\Spec(\mc{O}^{\mr{h}})) \ar@{-->}[r] & G(\ol{\QQ}_{\lambda}).
}
\end{equation}
Indeed, by Lemma \ref{tame} $\rho_{\lambda, \ol{\QQ}}$ factors through $\pi_1^t(X_{\ol{\kappa(v)}}, \bar{x}) \cong \pi_1(X_{\mc{O}^{\mr{sh}}}, \bar{x})$ (this isomorphism follows from tame specialization: see \cite[Corollary A.12]{lieblich-olsson-pi1} and \cite[Footnote 3]{klevdal-patrikis:SVcompatiblearXiv}), and $\rho_{\lambda, x}|_{F^{\mr{h}}}$ is unramified, giving the two components of the factorization. The restriction of $\rho_{\lambda}$ to $X_{F^h}$ therefore factors through $\pi_1(X_{\mc{O}^{\mr{h}}}, \bar{x})$, as this group is isomorphic to $\pi_1(X_{\mc{O}^{\mr{sh}}}, \bar{x}) \rtimes \pi_1(\Spec(\mc{O}^{\mr{h}}))$, compatibly with the canonical maps from $\pi_1(X_{F^{\mr{h}}}, \bar{x})$: this last point follows by Lemma \ref{fundseq}, an analogue over a henselian DVR of the exact sequence linking arithmetic and geometric fundamental groups. We are done.\footnote{Details: For each $r$, let $Y_r \to X[1/N_{\lambda}]$ be the finite \'etale covering corresponding to the mod $\lambda^r$ reduction over a suitable finite extension of $\QQ_{\ell}$, and let $\ol{Y}_r$ be the normalization of $X'= X[1/N_{\lambda}] \cup X_w$ in the function field of $Y_r$. We must check $\ol{Y}_r \to X'$ is \'etale in the fiber above $w$. This is clearly equivalent to checking $\ol{Y}_r \times \mc{O}_{(w)} \to X'_{\mc{O}_{(w)}}$ is \'etale. We are given that $\ol{Y}_r \times \mc{O}^{\mr{h}} \to X'_{\mc{O}^{\mr{h}}}$ is \'etale. Thus for some finite \'etale (since $\mc{O}_{(w)}$ is a DVR, \'etale neighborhoods of the special point are explicit) $\mc{O}_{(w)} \to \mc{O}'$, $\ol{Y}_r \times \mc{O}' \to X'_{\mc{O}'}$ is \'etale (by a limit argument). Thus $\ol{Y}_r \to X'$ is \'etale by \'etale descent.}

\end{proof}
\begin{lem}\label{fundseq}
Let $\mc{O}$ be a henselian DVR, and let $X \to \Spec(\mc{O})$ be a quasi-projective scheme such that $X_{\mc{O}^{\mr{sh}}}$ is connected. Let $\bar{x} \in X(\ol{K})$ be a geometric point with $K= \mr{Frac}(\mc{O})$. Then the canonical maps induce a short exact sequence
\[
1 \to \pi_1(X_{\mc{O}^{\mr{sh}}}, \bar{x}) \xrightarrow{\alpha} \pi_1(X, \bar{x}) \xrightarrow{\beta} \pi_1(\Spec(\mc{O}), \bar{x}) \to 1.
\]
\end{lem}
\begin{proof}
We imitate \cite[\href{https://stacks.math.columbia.edu/tag/0BTX}{Lemma 0BTX}]{stacks-project}. Surjectivity of $\beta$ is clear: connected objects of $\Fet(\mc{O})$ are all isomorphic to finite \'etale subextensions $\mc{O} \to \mc{O}' \subset \mc{O}^{\mr{sh}}$, and $X_{\mc{O}'}$ is then connected by assumption. The composite $\beta \circ \alpha$ is trivial; note that for $\mc{O}'$ as before, $\mc{O'} \otimes_{\mc{O}} \mc{O}^{\mr{sh}} \equiv \prod_{\Gal(K'/K)} \mc{O}^{\mr{sh}}$ (with $K'= \mr{Frac}(\mc{O}')$, $K= \mr{Frac}(\mc{O})$), just as for fields. To show the image of $\alpha$ is normal, we check that for $U \to X$ finite \'etale with $U$ connected and such that $U_{\mc{O}^{\mr{sh}}} \to X_{\mc{O}^{\mr{sh}}}$ has a section $s$, $U_{\mc{O}^{\mr{sh}}}$ is a finite coproduct of copies of $X_{\mc{O}^{\mr{sh}}}$. There is an $\mc{O}'$ as above such that $s$ is defined over $\mc{O}'$, and $s(X_{\mc{O}'})$ is an open connected component of $U_{\mc{O}'}$ (being a section of an \'etale separated morphism). Now consider
\[
V'= \bigcup_{\sigma \in \Gal(K'/K)} s^{\sigma}(X_{\mc{O}'}) \subset U_{\mc{O}'}.
\]
$V'$ has a canonical $\Gal(K'/K)$-action compatible with that on $\mc{O}'$, and $\mc{O} \to \mc{O}'$ is a Galois cover with group $\Gal(K'/K)$, so by Galois descent there is an open connected component $V$ of $U$ descending $V'$. Since $U$ is connected, $V=U$, and thus $U_{\mc{O}^{\mr{sh}}}= V_{\mc{O}_{\mr{sh}}}$ is a disjoint union of copies of $X_{\mc{O}^{\mr{sh}}}$. Injectivity of $\alpha$: any finite \'etale $V \to X_{\mc{O}^{\mr{sh}}}$ arises from $V' \to X_{\mc{O}'}$ for some finite \'etale $\mc{O} \to \mc{O}'$ as above. Then $V' \to X$ is still finite \'etale, and $V' \times_{\mc{O}} \mc{O}^{\mr{sh}}$ contains $V= V' \times_{\mc{O}'} \mc{O}^{\mr{sh}}$ as an open and closed subscheme (again using the Galois property $\mc{O}' \otimes_{\mc{O}} \mc{O}^{\mr{sh}} \cong \prod_{\Gal(K'/K)} \mc{O}^{\mr{sh}})$. Finally, to show exactness in the middle it remains to check that for any finite \'etale $U \to X$ such that $U_{\mc{O}^{\mr{sh}}}$ is a finite disjoint union of copies of $X_{\mc{O}^{\mr{sh}}}$, there is a finite \'etale $S \to \Spec(\mc{O})$ and a surjection $X \times_{\mc{O}} S \to U$. We first find a finite \'etale $\mc{O} \to \mc{O}'$ as above such that $U_{\mc{O}'}= \sqcup_{i=1}^n X_{\mc{O'}}$. $S= \sqcup_{i=1}^n \Spec(\mc{O}')$ works.
\end{proof}
\begin{cor}\label{svindram}
For a Shimura datum $(G, X)$ such that $Z_G(\QQ)$ is a discrete subgroup of $Z_G(\mathbb{A}_f)$, a neat level structure $K_0$, and a basepoint $s \in \mr{Sh}(\CC)$, there is an integer $N$ such that for all $\ell$ the canonical local system
\[
\rho_{\ell} \colon \pi_1(S_{K_0, s}, \bar{s}) \to G(\QQ_{\ell})
\]
is unramified outside $N\ell$, i.e. extends to $\mc{S}_{K_0, s}[1/\ell]$ for  an integral model $\mc{S}_{K_0, s}$ over $\mc{O}_{E_{K_0, s}}[1/N]$.
\end{cor}
In particular, this completes the proof in the non-adjoint case of \cite[Lemma 5.1]{bakker-shankar-tsimerman:canonicalmodels}, specifically the part deduced from Theorem 5.2 of \textit{loc. cit.}.

\section{Shimura Varieties}\label{SVsection}
Set $\mc{S}= \mc{S}_{K_0, s}$, our integral model over $\mc{O}_{E_{K_0, s}}[1/N]$. Here $N$ is chosen large enough that the conclusion of Corollary \ref{svindram} holds for $N$, and that \cite[Theorem 1.3]{bakker-shankar-tsimerman:canonicalmodels} holds away from this $N$, i.e. for $v \nmid N$, $\mc{S}_{K_0} \times_{\mc{O}_E[1/N]} \mc{O}_{E_v}$ is an integral canonical model.
We have the canonical local systems
\[
\rho_{\ell} \colon \pi_1(\mc{S}[1/\ell]) \to G(\QQ_{\ell}),
\]
which we will show in this section form a compatible system. For places $\lambda \vert \ell$, $\lambda' \vert \ell'$ of $\ol{\QQ}$, we also denote $\rho_{\ell}$ (resp. $\rho_{\ell'}$) by $\rho_{\lambda}$ (resp. $\rho_{\lambda'}$) when viewed as valued in $G(\ol{\QQ}_{\lambda})$ (resp. $G(\ol{\QQ}_{\lambda'})$. Let $v$ be a prime of $E_{K_0, s}$ not dividing $N \ell \ell'$. We form $\rho_{\lambda, v}$, the restriction to the $v$-fiber. Since the monodromy groups of the $\rho_{\lambda, v}$ might not be semisimple, we need a small argument to check that Drinfeld's results (\cite{drinfeld:pross}) apply. As explained in Corollary \ref{companions}, for the results of \cite[\S 6]{drinfeld:pross} to apply to give a $\lambda'$-companion, it suffices to check that the eigenvalues of $r(\rho_{\lambda, v})(\Frob_x)$ are plain of characteristic $p$ for all $x \in |\mc{S}_{\kappa(v)}|$ and representations $r$ of the algebraic monodromy group $G_{\lambda, v} \subset G$ of $\rho_{\lambda, v}$. To check this, it suffices to check that the composite $\pi_1(\mc{S}_v) \to G/G^{\mr{der}}$ is plain of characteristic $p$ (by \cite[Proposition 4.3]{chin:indl} or Lemma \ref{piplainlem}), but this follows from the canonical model property on the zero-dimensional Shimura variety associated to $G/G^{\mr{der}}$, and in particular is a special case of Lemma \ref{abelian} below. Thus a strong $\lambda'$-adic companion $\rho_{\lambda \leadsto \lambda', v}$ exists. 

We recall the following independence-of-$\ell$ calculation for special points:

\begin{lem}[Lemma 5.5 of \cite{huryn-kedlaya-klevdal-patrikis}]\label{abelian}
Let $s= [x, a] \in \mr{Sh}(\CC)$ be a special point defined in $S_{K_0, s}$ over a number field $E(sK_0)$. Then the canonical local system $\rho_{K_0, s}$ specialized to $sK_0$ defines a compatible system of $G(\QQ_{\ell})$-representations in the following sense: for any choice of Frobenius $\Frob_v$ at a place $v \vert p$ ($p$ not dividing $N$ as above) of $E(sK_0)$, $\rho_{\ell, s}(\Frob_v) \in G(\QQ_{\ell})$ is for all $\ell \neq p$ $G(\QQ_{\ell})$-conjugate to an independent of $\ell$ value $q_{\Frob_v}^{-1} \in G(\QQ)$. When $\ell=p$, as we vary the finite-dimensional representation $\xi$ of $G$, the F-isocrystal $D_{\mr{cris}}(\xi \circ \rho_{p, s}|_{\Gamma_{\Es_v}})$ on $\Spec(\kappa(v))$ with its linearized crystalline Frobenius $\varphi^{[\kappa(v):\mathbb{F}_p]}$ defines up to conjugacy an element $\gamma_v \in G(\ol{\QQ}_p)$, and $\gamma_v= q_{\Frob_v}^{-1}$ in $[G \git G](\QQ)$.
\end{lem}

To show $\rho_{\lambda \leadsto \lambda', v}$ is conjugate to $\rho_{\lambda', v}$, we will use results of \cite{bakker-shankar-tsimerman:canonicalmodels} that depend on the adjoint case of Theorem \ref{mainthm}, established in \cite{klevdal-patrikis:SVcompatiblearXiv}. We begin by reviewing these.

\begin{thm}[Theorem 1.6 of \cite{bakker-shankar-tsimerman:canonicalmodels}]\label{bst:CMlift}
Let $v$ be a prime of $E_{K_0, s}$ not dividing $N$. For any $x \in \mc{S}_{K_0}(\mathbb{F}_q)$ lying in the $\mu$-ordinary locus, there is a (canonical) special point $\tilde{x} \in \mc{S}_{K_0}(W(\mathbb{F}_q))$ lifting $x$.
\end{thm}
In particular, if $x$ lies in $\mc{S}(\mathbb{F}_q)$, then we obtain the lift $\tilde{x} \in \mc{S}(W(\mathbb{F}_q))$. (Recall $\mc{S}:= \mc{S}_{K_0, s}$.) Note that by choice of an isomorphism $W(\mathbb{F}_q) \subset \ol{\QQ}_p \xrightarrow[\sim]{\iota} \CC$, $\tilde{x}$ then becomes a classical special point; from \cite[Theorem 4.1]{pila-shankar-tsimerman:andre-oort} it is defined over $\ol{\ZZ}[1/N]$ and, being invariant under $\Gal(\ol{\QQ}_p/\mr{Frac}(W(\mathbb{F}_q))$ (identified via $\iota$ with a finite-index subgroup of the decomposition group of the place of $\ol{\QQ} \subset \CC$ induced by $\iota$) therefore is defined over $\mc{O}_L[1/N]$ for some number field $L$ with an unramified prime $w \vert p$ such that $\mc{O}_{L_w}= W(\mathbb{F}_q)$ inside $\ol{\QQ}_p$. In particular, for all $\lambda$ not above $p$, $\rho_{\lambda, v}(\Frob_x)$ is conjugate to $\rho_{\lambda}|_{\tilde{x}_L}(\Frob_w)$, where we write $\tilde{x}_L \in \mc{S}(\mc{O}_L[1/N])$ for the integral point underlying $\tilde{x}$ and $\rho_{\lambda}|_{\tilde{x_L}}$ for the specialization along $\tilde{x}_L$ (c.f. the specialization diagram in the proof of \cite[Theorem 3.10]{klevdal-patrikis:SVcompatiblearXiv}). Applying Lemma \ref{abelian}, we obtain:

\begin{cor}
Let $v \vert p$ be a prime of $E_{K_0, s}$ not dividing $N$. For any closed point $x \in \mc{S}_{\kappa(v)}$ belonging to the $\mu$-ordinary locus and all $\ell \neq p$, $\rho_{\ell}(\Frob_x)$ is $G(\QQ_{\ell})$-conjugate to an independent-of-$\ell$ element of $G(\QQ)$.
\end{cor}

\begin{thm}\label{mainthm}
For all finite places $\lambda$ and $\lambda'$ of $\ol{\QQ}$ and all $v$ not dividing $N \ell \ell'$, $\rho_{\lambda', v}= \rho_{\lambda \leadsto \lambda', v}$ as $G^{\mr{der}}(\ol{\QQ}_{\lambda'})$-conjugacy classes of homomorphisms $\pi_1(\mc{S}_{\kappa(v)}) \to G(\ol{\QQ}_{\lambda'})$.
\end{thm}
\begin{proof} 
Let $H$ be the algebraic monodromy group of $\rho_{\lambda', v}$. Our setup implies that $H$ lies between $G^{\mr{der}}$ and $G$ (but may be neither semisimple nor connected). From the known adjoint case (\cite[Theorem 3.10]{klevdal-patrikis:SVcompatiblearXiv}) and the abelian compatibility (Lemma \ref{abelian}), after suitable conjugation we may assume $\rho_{\lambda \leadsto \lambda', v}= \rho_{\lambda', v} \cdot \chi$ for some character (depending on $\lambda$, $\lambda'$, and $v$) $\chi \colon \pi_1(\mc{S}_{\kappa(v)}) \to Z_{G^{\mr{der}}}(\ol{\QQ}_{\lambda'})$. We know from Theorem \ref{bst:CMlift} that $\rho_{\lambda \leadsto \lambda', v}(\Frob_x)$ and $\rho_{\lambda', v}(\Frob_x)$ are conjugate for closed points $x$ in the Dirichlet density 1 (because Zariski open and dense: see Lemma \ref{density}) $\mu$-ordinary locus of $\mc{S}_{\kappa(v)}$. In particular, for any irreducible representation $R \colon G \to \mr{GL}_N$, by the \v{C}ebotarev density theorem (and Brauer-Nesbitt theorem), there is $g \in \mr{GL}_N(\ol{\QQ}_{\lambda'})$ such that 
\[
g(R \circ \rho_{\lambda', v}) g^{-1}= R \circ \rho_{\lambda \leadsto \lambda', v}= R \circ \rho_{\lambda', v} \cdot R\circ \chi.
\] 
In particular, $\tr(R(\rho_{\lambda', v}))= \tr(R(\rho_{\lambda', v})) \cdot R(\chi)$, where here we view $R(\chi)$ as a scalar (by Schur's lemma) in $\ol{\QQ}_{\lambda'}^\times$. Now, either $R(\chi)(\Frob_x)=1$ for a density 1 set of closed points $x$, in which case $R(\chi) =1$, or there is a positive density set of $x$ such that $\tr(R(\rho_{\lambda', v})(\Frob_x))$ equals zero. In the latter case, by \cite[Theorem 3]{rajan:sm1}, there is a connected component $\Phi$ of the algebraic monodromy group $H_R$ of $R(\rho_{\lambda', v})$ on which the trace vanishes identically.\footnote{Compare \cite[Proposition 3.4.9]{stp:variationsmemoir}. Strictly speaking, Rajan proves his Theorem for representations of $\pi_1(\mc{O}_F[1/N])$, $F$ a number field, but the same arguments will apply for integral finite type $\ZZ$-schemes, where we still have the \v{C}ebotarev density theorem.} For some $z \in Z_G$, $\Phi= R(z H^0)$ ($H$ surjects onto $H_R$, $H^0$ surjects onto $H_R^0$, and every element of $H/H^0$ is represented by an element of $Z_G$ since $H \supset G^{\mr{der}}$), and since $R(z)$ is a scalar, $\tr(\Phi)=0$ forces $\tr(R(H^0))=0$; this is clearly impossible, since $1 \in H^0$. We conclude that $R(\chi)=1$. Knowing this for all irreducible representations $R$ of $G$, we conclude that $\chi=1$.

\end{proof}
Here is the simple lemma cited in the course of the proof of Theorem \ref{mainthm}:
\begin{lem}\label{density}
Let $X$ be a scheme of finite-type over $\ZZ$. Let $U \subset X$ be a Zariski open and dense subset. Then the closed points $|U|$ have Dirichlet density 1 in $X$.
\end{lem}
\begin{proof}
For all such $X$, $\zeta_X(s)= \sum_{x\in |X|} (|\kappa(x)|)^{-s}$ converges absolutely and uniformly in $\Re(s)>d$ but has a pole at $s=d$ (the former statement is elementary, while the latter uses at least the Riemann hypothesis for curves over finite fields). If $U \subset X$ is open and Zariski-dense, then $Z= X \setminus U$ has dimension less than $d$, so $\zeta_Z(s)$ converges in $\Re(s)>d-1$. The density of $|U|$ is by definition 
\[
\lim_{s \to d^+} \frac{\sum_{x \in |U|} (|\kappa(x)|)^{-s}}{\zeta_X(s)}= \lim_{s \to d^+} \frac{\zeta_X(s)- \zeta_Z(s)}{\zeta_X(s)}= 1.
\]
\end{proof}

We next address the crystalline compatibility. The proof uses the same principle as the proof of Theorem \ref{mainthm}, except we need the existence of $p$-to-$\ell$ companions in our setting, as explained in \cite{huryn-kedlaya-klevdal-patrikis} in the adjoint case and in Corollary \ref{companions} (or Corollary \ref{plainindl}) in general. By \cite[Theorem 7.1]{pila-shankar-tsimerman:andre-oort}, 
for any place $v \vert p$ of $E_{K_0, s}$ not dividing $N$  the $p$-adic local system $\rho_{p, v} \colon \pi_1(\mc{S}_{(E_{K_0, s})_v}) \to G(\QQ_p)$ is crystalline. Our joint work \cite{huryn-kedlaya-klevdal-patrikis} explains how to extract from this an ($G$-) overconvergent $F$-isocrystal, packaged as a $G(\ol{\QQ}_p)$-conjugacy class of representations
\[
\rho_{p, v}^{\FIso^{\dagger}} \colon \pi_1^{\FIso^{\dagger}}(\mc{S}_v) \to G(\ol{\QQ}_p)
\]
compatible (via $D_{\mr{cris}}$) with $\rho_{p, v}$. For notational consistency, we will also for a place $\lambda$ of $\ol{\QQ}$ above $p$ write these representations as $\rho_{\lambda, v}^{\FIso^{\dagger}}$ and $\rho_{\lambda, v}$. We further know (\cite[Proposition 4.3.1]{huryn-kedlaya-klevdal-patrikis}) that the image of $\rho_{\lambda, v}^{\FIso^{\dagger}}$ in $G^{\mr{ad}}$ is Zariski-dense, and from Lemma \ref{abelian} we see that for any closed point $x \in \mc{S}_v$ that lifts to a special point, the linearized crystalline Frobenius of $\rho_{\lambda, v}^{\FIso^{\dagger}}|_{\kappa(x)}$ agrees in $[G \git G](\QQ)$ with $\rho_{\lambda', v}(\Frob_x)$ ($\lambda'$ not above $p$).\footnote{Here and in what follows, for $\lambda' \vert p$ there is also a $\lambda'$-companion, but since $\rho_{p, v}$ has $\QQ_p$-coefficients, the ``$p$-to-$p$" companion construction is just the identity.} 
This compatibility on special points also implies the compatibility of the abelianizations of $\rho_{\lambda, v}^{\FIso^{\dagger}}$ and $\rho_{\lambda', v}$. 
Combining work of Drinfeld (\cite{drinfeld:pross}), Abe (\cite{abe:crystallinecompanions}), Abe-Esnault (\cite{abe-esnault:crystallinecompanions}), and Kedlaya (\cite{kedlaya:crystallinecompanions1}), 
we have seen in Corollary \ref{companions} that $\rho_{\lambda, v}^{\FIso^{\dagger}}$ admits a $\lambda'$-companion $\rho_{\lambda \rightsquigarrow \lambda', v}$ whose algebraic monodromy group also contains $G^{\mr{der}}$.
\begin{cor}\label{crysthm}
For any place $v \vert p$ not above $N$ of $E'$, and all finite places $\lambda \vert p$ and $\lambda'$ of $\ol{\QQ}$, $\rho_{\lambda, v}^{\FIso^{\dagger}}$ and $\rho_{\lambda', v}$ ($\lambda'$ not above $p$) 
are $G$-companions.
\end{cor}
\begin{proof}
The adjoint case is the main result of \cite{huryn-kedlaya-klevdal-patrikis}. By the observations in the paragraph preceding the Corollary (existence of the crystalline-\'etale companions on $\mc{S}_v$ and the previously-known, by Lemma \ref{abelian}, compatibility of $\rho_{\lambda, v}^{\FIso^{\dagger}}$ and $\rho_{\lambda', v}$ on special points), the proof of Theorem \ref{mainthm} now applies verbatim. 
\end{proof}
We conclude with a refinement of Corollary \ref{crysthm} that supplies one of the desiderata in the production of Kottwitz-triples (see the introduction). Let $v \vert p$ be a place not above $N$ of $E'$, and let $x \in \mc{S}_v$ be a closed point. By construction, for all $\lambda'$ not above $p$, $\rho_{\lambda'}(\Frob_x)$ lies in $G(\QQ_{\ell'})$, is semisimple by \cite[Theorem 7.1]{bakker-shankar-tsimerman:canonicalmodels}, and we have seen this is an independent of $\lambda'$ conjugacy class in $[G \git G](\QQ)$. For $\lambda \vert p$, let us write $\rho_{\lambda, v}^{\FIso^{\dagger}}(\Frob_x) \in G(\ol{\QQ}_p)$ for the linearized crystalline Frobenius of $\rho_{\lambda, v}^{\FIso^{\dagger}}|_{\kappa(x)}$. Corollary \ref{crysthm} tells us the semisimple part $\gamma_x$ of $\rho_{\lambda, v}^{\FIso^{\dagger}}(\Frob_x) \in G(\ol{\QQ}_p)$ is $G(\ol{\QQ}_p)$-conjugate to an element of $G(\ol{\QQ})$ whose conjugacy class is $\QQ$-rational (induced by $\rho_{\lambda'}(\Frob_x)$). Note that we only at present know that $\rho^{\FIso^{\dagger}}_{\lambda, v}(\Frob_x)$ is semisimple for $x$ in the $\mu$-ordinary locus, since $x$ then lifts to a special point in characteristic zero.
\begin{cor}
Let $K_{x}= \mr{Frac}(W(\kappa(x)))$. There is an element $\delta \in G(K_x)$, defined up to $\sigma$-conjugacy by elements of $G(W(\kappa(x)))$, such that the $G(\ol{\QQ}_p)$-conjugacy class of $\rho_{\lambda, v}^{\FIso^{\dagger}}(\Frob_x)$ equals that of $\delta \sigma(\delta) \cdots \sigma^{d_x-1}(\delta)$, where $d_x= [\kappa(x):\mathbb{F}_p]$.
\end{cor}
\begin{proof}
Recall that $K_p= G_{\ZZ_p}(\ZZ_p)$ is a hyperspecial level structure, where we write $G_{\ZZ_p}$ for an extension of $G$ to a reductive model over $\ZZ_p$. Choose a faithful representation (over $\ZZ_p$) $L$ of $G_{\ZZ_p}$ and set $V= L \otimes_{\ZZ_p} \QQ_p$. Any choice of lift $\tilde{x} \in \mc{S}(W(\kappa(x)))$ induces on $L$ the structure of a crystalline $\pi_1(K_x)$-representation, with underlying $F$-isocrystal independent of the choice of lift $\tilde{x}$. By \cite[Proposition 1.3.2]{kisin:intmodabelian}, there is a finite set of tensors $(s_{\alpha}) \subset L^\otimes$ (the direct sum of all $\ZZ_p$-modules formed from $L$ using duals, tensor products, and symmetric and exterior powers) such that $G_{\ZZ_p}$ is the subgroup of $\mr{GL}(L)$ fixing (pointwise) all $s_{\alpha}$. Let $\mathfrak{M}$ be Kisin's fully faithful tensor functor (\cite{kisin:crystalline}, \cite[\S 1.2]{kisin:intmodabelian})
\[
\mathfrak{M} \colon \Rep_{\pi_1(K_x)}^{\mr{cris}, \circ} \to \mr{Mod}^{\varphi}_{/\mathfrak{S}}
\] 
from lattices in crystalline representations of $\pi_1(K_x)$\footnote{Note Kisin's theory works with ramified ground fields as well.} to finite free $\mathfrak{S}=W(\kappa(x))[[u]]$-modules equipped with an isomorphism $\varphi^*(\mathfrak{M})[1/E(u)] \xrightarrow{\sim} \mathfrak{M}[1/E(u)]$; here $E(u)$ is an Eisenstein polynomial defining a fixed uniformizer of $K_x$, so we can just take $E(u)=u-p$.

The tensors $s_{\alpha} \in L^\otimes$ transport via $\mathfrak{M}$ to tensors $(\tilde{s}_{\alpha}) \subset \mathfrak{M}(L)^{\otimes}$; Galois-invariance of the $s_{\alpha}$ yields $\varphi$-invariance of the $\tilde{s}_{\alpha}$, and then \cite[Corollary 1.3.5]{kisin:intmodabelian} shows that the $(\tilde{s}_{\alpha})$ define a reductive group $G_{\mathfrak{M}} \subset \mr{GL}(\mathfrak{M}(L))$ isomorphic to $G_{\ZZ_p} \times_{\ZZ_p} \mathfrak{S}$. More precisely, the isomorphism arises from trivializing the $G_{\ZZ_p} \times_{\ZZ_p} \mathfrak{S}$-torsor $\mc{P}:= \underline{\mr{Isom}}_{\mathfrak{S}}(L \otimes_{\ZZ_p} \mathfrak{S}, \mathfrak{M}(L))$ of linear isomorphisms mapping $s_{\alpha}$ to $\tilde{s}_{\alpha}$. 
Specializing along $u \mapsto 0$, we obtain tensors $\tilde{s}_{\alpha, 0} \in (\mathfrak{M}(L)/u\mathfrak{M}(L))^{\otimes}$ and a corresponding trivialization.

Thus for a choice $\gamma$ of trivialization over $\mathfrak{S}/u\mathfrak{S}=W(\kappa(x))$, we have the $\sigma$-linear injection
\[
\Phi \colon L \otimes_{\ZZ_p} W(\kappa(x)) \xrightarrow[\sim]{\gamma} \mathfrak{M}(L)/u\mathfrak{M}(L) \xrightarrow{\varphi} \mathfrak{M}(L)/u\mathfrak{M}(L)[1/p] \xrightarrow[\sim]{\gamma^{-1}} L \otimes_{\ZZ_p} K_x
\]
that carries $s_{\alpha}$ to $s_{\alpha}$, since the tensors $\tilde{s}_{\alpha}$ and thus $\tilde{s}_{\alpha, 0}$ are $\varphi$-invariant.  The composite $\Phi$ above can therefore be uniquely expressed as $b \sigma$ for some $b \in G(K_x)$. Changing the choice of trivialization of $\mc{P} \pmod{u}$ amounts to replacing $\gamma$ by $\gamma s$, $s \in G(W(\kappa(x)))$, which changes $\Phi$ to $s^{-1}\Phi s$ and changes $b$ to $s^{-1}b \sigma(s)$. Finally, $\Phi^{d_x}=b \sigma(b) \cdots \sigma^{d_x-1}(b)$ is an element of $G(K_x)$ in the same $G(\ol{\QQ}_p)$-conjugacy class as $\rho^{\FIso^{\dagger}}_{\lambda, v}(\Frob_x)$, since $\mathfrak{M}(L)/u\mathfrak{M}(L)[1/p]$ is canonically $D_{\mr{cris}}(V)$ as $\varphi$-isocrystals (\cite[Theorem 1.2.1]{kisin:intmodabelian}).
\end{proof}

\bibliographystyle{amsalpha}
\bibliography{biblio.bib}

\end{document}